\documentclass[11pt]{amsart}
\usepackage{amssymb, amsmath}
\usepackage{amsthm}
\usepackage{amscd}
\usepackage{graphicx}
\newtheorem{theorem}{Theorem}
\newtheorem{lemma}[theorem]{Lemma}

\newtheorem{conjecture}[theorem]{Conjecture}
\theoremstyle{definition}

\newtheorem*{acknowledgement}{Acknowledgement}

\title{Minimal orbifolds and (a)symmetry of piecewise locally symmetric manifolds}
\author{T. T$\hat{\mathrm{a}}$m Nguy$\tilde{\hat{\mathrm{e}}}$n Phan}

\DeclareMathOperator{\Vol}{Vol}
\DeclareMathOperator{\Ker}{Ker}

\DeclareMathOperator{\Isom}{Isom}
\DeclareMathOperator{\Out}{Out}

\DeclareMathOperator{\SL}{SL}
\DeclareMathOperator{\SO}{SO}
\DeclareMathOperator{\CAT}{CAT}

\DeclareMathOperator{\Fix}{Fix}

\def\R{\mathbb{R}}
\def\Z{\mathbb{Z}}

\def\Q{\mathbb{Q}}

\def\X{\overline{X}}
\def\M{\overline{M}}

\input xy
\xyoption{all}

\begin{document}
\begin{abstract}
We show that if $g$ is a Riemannian metric on a closed piecewise locally symmetric manifold $M$, then the lift of $g$ to the universal cover $\widetilde{M}$ has  a discrete isometry group. We also show that the index $[\Isom(\widetilde{M}): \pi_1(M)]$ is bounded by a constant independent of $g$.
\end{abstract}
\maketitle
\section{Introduction}

Let $M^n$ be a smooth manifold of dimension $n$, and let $g$ be a Riemannian metric on $M$.  Let $\widetilde{g}$ be the lifted metric on the universal cover $\widetilde{M}$. To study the amount of symmetry of $M$, one traditionally studies the group $\Isom(M,g)$ of isometries of $(M,g)$. While nonpositively curved, locally symmetric spaces of noncompact type have discrete isometry groups, they are ``highly symmetric", i.e. their universal covers are symmetric manifolds and have continuous isometry groups. With this perspective on symmetry of a Riemannian manifold, a more interesting way of measuring the symmetry of $M$ is the group $\Isom(\widetilde{M},\widetilde{g})$, for it contains symmetries that are ``hidden" on covers of $M$.  

The group $\pi_1(M)$ of deck transformations on $\widetilde{M}$ is a subgroup of $\Isom((\widetilde{M}), \widetilde{g})$. A natural question to ask is how large can the index $[\Isom(\widetilde{M}): \pi_1(M)]$ be. The following phenomenon is interesting.  

If $M$ is diffeomorphic to a locally symmetric manifold, and $g$ is some Riemannian metric, then there is a constant $C(\pi_1(M))$ such that if $[\Isom(\widetilde{M}): \pi_1(M)] > C$, then $g$ is scalar multiple of the locally symmetric metric, in which case $[\Isom(\widetilde{M}: \pi_1(M)] = \infty$ and $\Isom(\widetilde{M})$ is a Lie group of positive dimension. The case where $M$ is closed is due to Farb and Weinberger (\cite{FW2}), and the case where $M$ is noncompact is due to Avramidi (\cite{Avramidistucktori}).  

The mechanism for such a finite threshold for non-discreteness  of $\Isom(M)$ is the existence of minimal orbifolds covered by $M$. More precisely, the question of existence of such minimal orbifolds can be formulated as follows. An orbifold $V$ that is (orbifold) covered by $M$ is a \emph{minimal orbifold} if it does not cover any other orbifold. One can ask which manifold $M$ has a minimal orbifold?     

The existence of such minimal orbifolds for locally symmetric manifolds is proved by Farb and Weinberger (\cite{FW2}) by combining Mostow-Prasad-Margulis rigidity, a theorem of Kazhdan and Margulis, and a theorem of Borel (without Margulis). This does not hold, however, for general locally homogeneous manifolds. For example, compact nilmanifolds (or more generally circle bundles) do not have minimal orbifolds. It will be interesting to find a proof of the existence of minimal orbifolds that does not depend on the above three ingredients. Another class of spaces that admit minimal orbifolds is moduli spaces of curves by results of Avramidi(\cite{AvramidiL2}). It is not known that there other any other manifolds that admit mininal orbifolds.

In this paper, we prove that closed \emph{piecewise locally symmetric manifolds} admit minimal orbifolds. These manifolds are aspherical and were introduced in \cite{Tampwlocsym}. We will recall the definition of these manifolds in the next section. As the name suggests, piecewise locally symmetric manifolds decompose into pieces that are diffeomorphic to locally symmetric manifolds.

%Let $M$ be a closed piecewise locally symmetric manifold, and let $g$ be a Riemannian metric on $M$. The fundamental group $\pi_1(M)$ is centerless (\cite{Tampwlocsym}). By a theorem of Borel (\cite{CR}) and since $M$ is aspherical, $M$ does not have distinct homotopic isometries. In other words, $\Isom(M,g)$ embeds as a subgroup of  the group $\Out(\pi_1(M))$ of self homotopy equivalences. In particular $\Isom(M,g)$ is discrete. 

The main theorem of this paper is the following.
\begin{theorem}\label{magic number}
Let $M$ be a closed piecewise locally symmetric manifold. There exists a constant $C$ depending only on $\pi_1(M)$ such that for any metric $h$ on $M$, we have \[ [\Isom(\widetilde{M}): \pi_1(M)] \leq C.\] 
The constant $C$ can be taken to be $\Vol(M)/V_n$, where $V_n$ is the constant in the Kazhdan-Margulis theorem for the symmetric space corresponding to the pieces of $M$.  
%The equality is assumed if $h$ is the singular hyperbolic metric $g_s$ on $M$.
\end{theorem}
Geometrically, the above theorem says two things. Firstly, closed piecewise locally symmetric manifolds have at most a discrete amount of symmetry, i.e. $\Isom(\widetilde{M})$ is discrete. In particular, closed piecewise locally symmetric manifolds do not admit any locally homogeneous metric. Secondly, the degree of symmetry of a given closed piecewise locally symmetric manifold is universally bounded independently from the metric. This is because piecewise locally symmetric manifolds admit minimal orbifolds.

All of the above examples are special cases of a conjecture of Farb and Weinberger (\cite[Conjecture 1.6]{FW2}).
\begin{conjecture}[Farb and Weinberger]\label{magic number conj}
Let $M$ be any closed Riemannian $n$--manifold, $n > 1$. There exists a constant $C$ depending only on $\pi_1(M)$ such that the following are equivalent:

1. $M$ is aspherical, smoothly irreducible (i.e. $M$ is not finitely covered by a product), $\pi_1(M)$ has no non-trivial, normal abelian subgroup and $[\Isom(\widetilde{M}) \colon \pi_1(M)] > C$.

2. $M$ is isometric to an irreducible, locally symmetric Riemannian manifold of non-positive sectional curvature.
\end{conjecture}

\begin{acknowledgement} I would like to thank my advisor, Benson Farb, for suggesting this problem and for his guidance. I would like to thank Grigori Avramidi, Ben McReynolds and Shmuel Weinberger for helpful conversations.
\end{acknowledgement}

\section{Piecewise locally symmetric manifolds: backgrounds and definition}

\label{sec:definitions}

We briefly recall the definition of \emph{piecewise locally symmetric manifolds} and state some of their properties, which we will use later to prove the main theorem. For more detailed exposition, the readers can see \cite{Tampwlocsym}.

\subsection{Compactifications of locally symmetric spaces}

Let $X$ be a connected, noncompact, finite-volume, complete, locally symmetric, nonpositively curved manifold. Then $X$ has a compactification $\overline{X}$ that is a manifold with corners as follows.

If $X$ has $\R$-rank equal to $1$, then $\overline{X}$ can be taken to be $X$ with deleted ends at a ``horo cross section" (see \cite{Tamrigidity}). Each end of $N$ is the quotient of a horoball in $\widetilde{X}$ by a group of parabolic isometries. Diffeomorphically, these are $S\times [0,\infty)$, for $S$ a compact infra-nilmanifold that is a quotient of a horosphere in $\widetilde{X}$.

If $X$ has $\R$-rank $\geq 2$, then by the Margulis arithmeticity theorem (\cite{Zimmer}), $X = K \backslash G/\Gamma$, for some semisimple, linear, connected algebraic group defined over the rational numbers $\Q$ with a maximal compact subgroup $K$, and $\Gamma$ is an arithmetic lattice of $G$. We take $\overline{X}$ to be the Borel-Serre compactification of $X$ (\cite{Ji}).

\subsection{Piecewise locally symmetric spaces}

Let $\overline{X}$ be as above. Let $X_i$, for $i \in S$, be the codimension $1$ strata of $\X$. For each $x \in \X$, let
\[ S(x) = \{ i \in S \: | \: x \in X_i\}.\]
Let $W$ be a Coxeter group whose generators are in one-to-one correspondence with the codimension $1$ strata of $\overline{X}$. With abused notation, we write $S = \{s_i\}$ be the set of generators of $W$. We require that for each subset $T \subset S$, the subgroup generated by elements in $T$ is finite if and only if the corresponding strata have nonempty intersection. We define $W_T$ to be the subgroup of $W$ that is generated by all $s\in T$.

Define and equivalence relation on $W\times \X$ as $(h,x)\sim(g,y)$ if $x=y$ and $h^{-1}g \in W_{S(x)}$. Then $U(W,\X)$ (defined in \cite{Davis}) is the quotient space $(W\times \X) /\sim$ and is a manifold. The group $W$ naturally acts on $U(W,\X)$ with quotient $\X$, and this action is proper. We call manifolds that are quotients of $U(W,\X)$ by subgroups of $W$ \emph{piecewise locally symmetric spaces}. These manifolds need not be compact, but one can obtain always compact such manifold, e.g. $U(W,\X)/\Gamma$, for some torsion free finite index subgroup $\Gamma$ of $W$. The existence of such $\Gamma$ is guaranteed since $W$ is linear.  

\bigskip
\noindent
\textbf{Examples:}
\begin{itemize}
\item[1)] Let $X$ be a noncompact, finite volume, hyperbolic manifold. Then each ends of $X$ is $S\times[0,\infty)$ for some compact flat manifold $S$. Let $\overline{X}$ be the manifold obtained by deleting $S\times (1,\infty)$. Let $M$ be a double of $X$. Then $M$ is a piecewise locally symmetric manifold. In this case, $W = \Z_2$.
\item[2)] Let $\Gamma$ be and torsion free finite index subgroup of  $\SL(3,\Z)$, and let $X = \SO(3)\backslash\SL(3,\R)/\Gamma$. The compactification $\X$ is a manifold with corners with codimension $1$ and codimension $2$ strata as below. See also Figure~\ref{SL(3,Z)} for an illustrating picture of $\M$.
\begin{figure}\label{SL(3,Z)}
\begin{center}
\includegraphics[height=100mm]{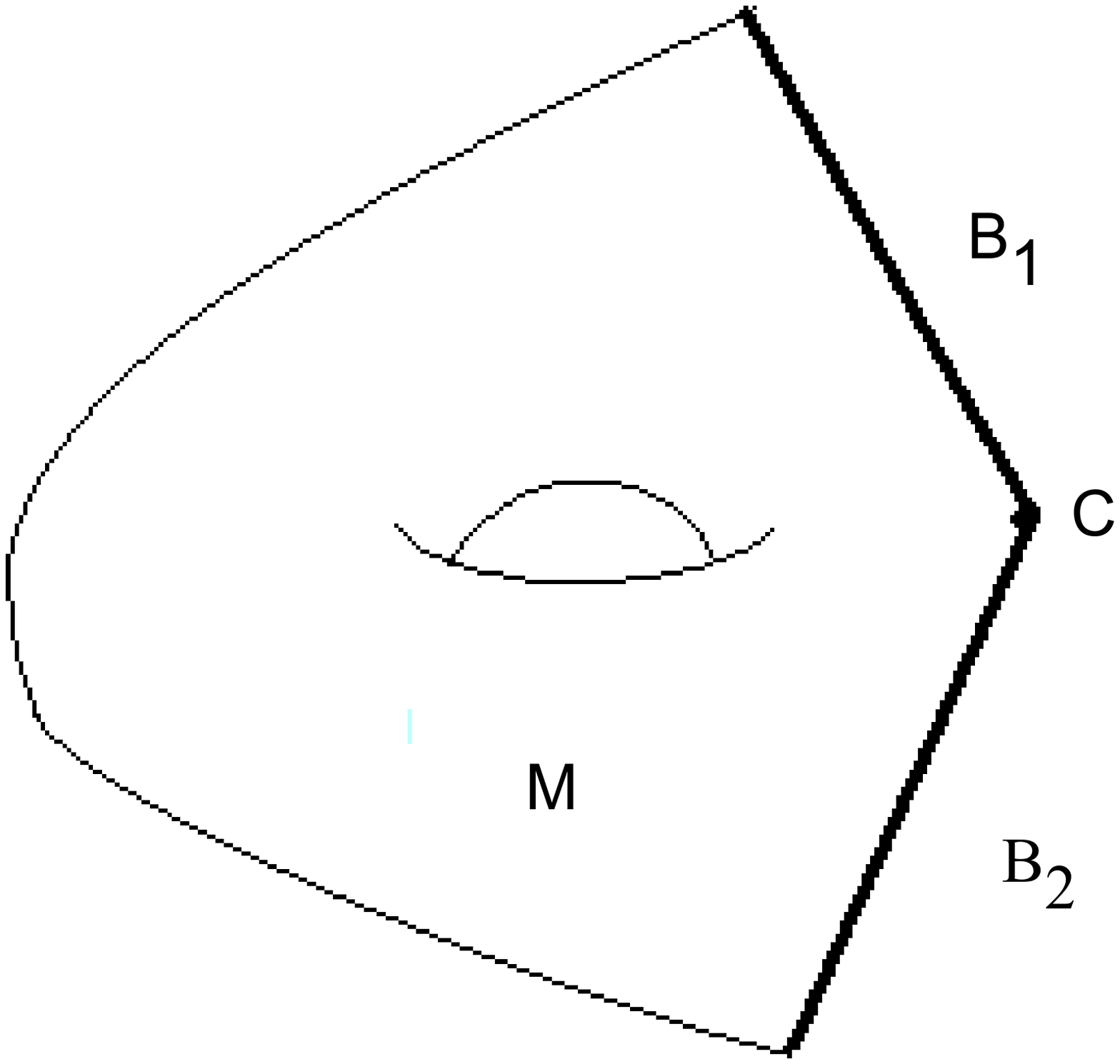}
\caption{Schematic of a manifold with corners with two codimension $1$ strata and one codimension $2$ stratum.}
\end{center}
\end{figure}
  
Recall that the parabolic subgroups of $\SL(n,\R)$ are precisely the subgroups that preserves a flag. There are two kinds of codimension $1$ strata, which correspond respectively to the $\Gamma$-conjugacy classes of the following two parabolic subgroups of $\SL(3,\R)$.
\[ P_1 =\left(\begin{matrix} *&*&*\\ *&*&*\\ 0&0&* \end{matrix}\right), \qquad P_2 =\left(\begin{matrix} *&*&*\\ 0&*&*\\ 0&*&* \end{matrix}\right),\]
and one kind of codimension $2$ stratum, which corresponds to the $\Gamma$-conjugacy class of the parabolic subgroup
\[Q =\left(\begin{matrix} *&*&*\\ 0&*&*\\ 0&0&* \end{matrix}\right).\]
For each $i = 1, 2$, the codimension $1$ strata are the following spaces 
\[B_i = (P_i \cap \SO(3))\backslash P_i/ (P_i\cap \gamma),\] 
which a $2$-torus bundle over $\SO(2)\backslash\SL(2,\R)/(\Gamma \cap\SL(2,\Z))$, and the codimension $2$ strata are of the form 
\[C = (Q\cap \SO(3))\backslash Q/ (Q\cap \Gamma),\] 
which is a compact nil-manifold. 

Label the codimension $1$ strata of $\X$ as $s_i$, for $i \in I$ (an index set). Let $W$ be the right angle Coxeter system with generators $s_i \in S$. The space $U(W,\X)$ is manifold that is union of copies of $\X$ glued to each other along pairs of codimension $1$ boundary strata. Now, let $\Gamma$ be a torsion free, finite index subgroup of $W$. The quotient $U(W,\X)/\Gamma$ is a compact piecewise locally symmetric manifold.
 
\end{itemize}

\subsection{Rigidity of piecewise locally symmetric spaces}   
Let $M$ be a compact piecewise locally symmetric manifold of dimension $n > 2$. The decomposition of $M$ into locally symmetric pieces are rigid under homotopy equivalence in the following sense (\cite{Tamrigidity}, \cite{Tampwrank1}, \cite{Tampwlocsym}).
\begin{theorem}\label{piece rigidity}
Let $M$ be a piecewise locally symmetric manifold of dimension $n > 2$, and let $M_i$, $i\in I$, be the locally symmetric pieces in the decomposition of $M$. Suppose that $M_i$ are irreducible. Let $f \colon M \longrightarrow M$ be a homotopy equivalence. Then restriction of $f$ to each $M_i$ is homotopic to a map $g \colon M_i \longrightarrow M$ that is a diffeomorphism onto a piece $M_j \subset M$. 
\end{theorem}

\section{Piecewise locally symmetric manifolds have discrete amount of symmetry}

\begin{theorem}\label{discrete Isom}
Let $M$ be a piecewise locally symmetric manifold of dimension $n \geq 3$ and let $h$ be a metric on $M$. Let $\widetilde{h}$ be the lifted metric of $h$ on the universal cover $\widetilde{M}$ of $M$. Then $\Isom(\widetilde{M},\widetilde{h})$ is discrete. 
\end{theorem}

\begin{proof}
The two key ideas of the proof are the following. Firstly, we prove that if $\Isom(\widetilde{M}, \widetilde{h})$ is not discrete, then by \cite[Theorem 1.3]{FW2}, the manifold $(M, h)$ is isometric an irreducible, locally symmetric Riemannian manifold of nonpositive sectional curvature of $\R$ rank greater than $1$. Secondly, we prove that $\pi_1(M)$ has an infinite index, infinite, normal subgroup and obtain a contradiction to the Margulis normal subgroup theorem.
  
Suppose that $\Isom(\widetilde{M},\widetilde{h})$ is not discrete. We check that $M$ satisfies the conditions in the first statement of \cite[Theorem 1.3]{FW2}. By \cite{Tampwlocsym}, the manifold $M$ is aspherical. We need to show that $M$ is smoothly irreducible and does not have any normal, infinite abelian subgroup.

In order to prove the above, we need to use the fact the $\pi_1(M)$ is the fundamental group of a nonpositively curved complex of groups since $M$ is a complex $T$ of spaces. This is established in \cite{Tampwlocsym}. What we needs in the following. The vertex groups are $\pi_1(M_i)$'s, the edge groups are the fundamental group of the codimension $1$ strata, etc. The universal cover $\widetilde{M}$ is a simply connected, $\CAT(0)$ complex $\widetilde{T}$ of spaces. The group $\pi_1(M)$ acts on $\widetilde{T}$ by isometries with quotient $T$. The stabilizer of each cell in $\widetilde{T}$ is a conjugate of the corresponding cell group.  

\begin{itemize}

\item \textit{Showing that $M$ is smoothly irreducible:}

\textbf{Case 1:} The $\R$ rank of $M_i$'s is $1$. If $M$ is finitely covered by $\widehat{M}$ that is a product, then the centralizer of each element in $\pi_1(\widehat{M})$ contains a $\Z^2$. By \cite[Lemma 5]{Tamrigidity} any subgroup of $\pi_1(\widehat{M})$ that is isomorphic to $\Z^2$ is conjugate to a subgroup in a cusp subgroup of $\widehat{M}$. But this implies that any element in $\pi_1(\widehat{M})$ belongs to a conjugate of a cusp subgroup, which is clearly not true. Hence, $M$ is smoothly irreducible (i.e. not finitely covered by a product).

\textbf{Case 2:} The $\R$-rank of $M_i$'s is $2$. If $M$ is finitely covered by $\widehat{M}$ that is a product $N_1\times N_2$. Let $\widehat{M}_a$ be a piece in the decomposition of $\widehat{M}$. The image of $\pi_1(\widehat{M}_a)$ under the projection $p \colon \pi_1(\widehat{M}) = \pi_1(N_1) \times \pi_1(N_2) \longrightarrow \pi_1(N_i)$ either is finite or has finite kernel since $\widehat{M}_a$ has higher rank, so the Margulis normal subgroup theorem applies $\widehat{M}_a$. 

If $p(\pi_1(\widehat{M}_a))$ is finite, then it is trivial since $\pi_1(N_i)$ is torsion free as $\pi_1(\widehat{M})$ is torsion free. Hence, $\pi_1(\widehat{M}_a)$ is a subgroup of $\pi_1(N_1)$ or $\pi_1(N_2)$. Without loss of generality, suppose that $\pi_1(\widehat{M}_a$ is a subgroup of $\pi_1(N_1)$. Then $\pi_1(N_2)$ commutes with $\pi_1(\widehat{M}_a)$, and thus, preserves the Fix set of $\pi_1(\widehat{M}_a)$ in the complex $\widetilde{T}$. The fixed set of $\pi_1(\widehat{M}_a)$ is a point since $\pi_1(\widehat{M}_a)$ is a vertex group. It follows that $\pi_1(N_2)$ has to fix the same point, so $\pi_1(N_2) \leq \pi_1(\widehat{M}_a)$. But the projection of $\pi_1(\widehat{M}_a)$ onto $\pi_1(N_2)$ is trivial, so $\pi_1(N_2)$ is trivial. Since $N_2$ is a compact, simply connected, aspherical manifold, it has to a point. Hence, $\widehat{M}$ is not a nontrivial product. Therefore, $M$ is smoothly irreducible.
\newline
\item \textit{Showing that $\pi_1(M)$ does not contain a nontrivial normal abelian subgroup $A$:}

Suppose it does. Being normal and nontrivial, the group $A$ is not conjugate to a subgroup of any vertex group of $\pi_1(M)$ for the following reasons. The group $A$ is infinite since it is nontrivial, abelian and torsion free. If $A$ is conjugate to a subgroup of a vertex group of $\pi_1(M)$, say, $\pi_1(M_a)$, then $\pi_1(M_a)$ has an infinite, normal subgroup, which must have finite index in $\pi_1(M_a)$ by the normal subgroup theorem. Hence, $\pi_1(M_a)$ is virtually abelian, which contradicts the fact that $M_a$ is a higher rank lattice.

Hence, $A$ does not fix a point in the complex $\widetilde{T}$ of $\pi_1(M)$.  By \cite[Proposition 2.3]{Farb}, there is an $A$-invariant flat $F$ in $\widetilde{T}$ on which $A$ acts by translations. Since $A$ is normal, $gag^{-1} \in A$ for all $g \in \pi_1(M)$. But $gag^{-1}$ is a translation along the flat $g(F)$, so it is not a element in $A$ for general $g$. Hence, $\pi_1(M)$ does not contain a nontrivial finitely generated normal abelian subgroup. 
\end{itemize}

We have shown that $(M, h)$ satisfies the conditions in the first statement of \cite[Theorem 1.3]{FW2}. Thus, $M$ is isometric to an irreducible, locally symmetric Riemannian manifold of nonpositive sectional curvature. Therefore, $M = \Gamma\setminus G/K$, where $G$ is a semisimple Lie group of noncompact type, $K$ is a maximal compact subgroup of $G$ and $\Gamma$ is an irreducible, cocompact lattice that is isomorphic to $\pi_1(M)$. Since $\pi_1(M)$ contains a subgroup that is isomorphic to $\Z^2$, it follows that $M$ must have higher $\R$ rank. Therefore, $\pi_1(M)$ is almost simple by the normal subgroup theorem. 

Let $X$ be a fundamental chamber of $M$. The inclusion of $X$ into $M$ is a retraction. There is a folding map $f \colon M \longrightarrow X$ that restricts to the identity on $X$. Hence the induced map $f_*$ on fundamental groups is surjective. Thus $H := \Ker f_*$ is an infinite index subgroup of $\pi_1(M)$. We claim that $H$ is infinite and thus contradict the Margulis normal subgroup theorem. This is easily seen by looking at the obvious elements that get sent to $1$ by $f_*$, i.e. the concatenation of a (hyperbolic type) loop in one copy of $X$ with its inverse in an adjacent copy of $X$ in $M$. This is element is a non-identity element, and thus, has infinite order since $\pi_1(M)$ is torsion free.  
\end{proof}

\section{Piecewise locally symmetric manifolds have a bounded amount of symmetry}
We will prove Theorem \ref{magic number} in this section. Firstly, we prove the following lemma which we will need in the proof of Theorem \ref{magic number}.
\begin{lemma}\label{most symmetry}
Let $M$ be a piecewise locally symmetric manifold. Then for any Riemannian metric $h$ on $M$, the group $\Isom(M,h)$ is isomorphic to a subgroup of $\Out(\pi_1(M))$.
\end{lemma}

\begin{proof}
By Theorem \ref{discrete Isom}, the group $\Isom(\widetilde{M}, h)$ is discrete, which implies that $\Isom(M, h)$ is discrete. Let $\phi \colon \Isom(M, h) \longrightarrow \Out(\pi_1(M))$ be the canonical homomorphism. We need to prove that $\phi$ is injective. But this follows from a  theorem in \cite{CR} which says the following.
\begin{theorem}[Borel, Conner, Raymond]
Let $M$ be a closed connected aspherical manifold with centerless fundamental group. If $G$ is a finite group that act effectively on $M$, then the canonical homomorphism $\psi \colon G \longrightarrow \Out(\pi_1(M))$ is a monomorphism.
\end{theorem}
Since, $M$ is closed connected and aspherical we just need to check that $\pi_1(M)$ has trivial center. But this follows from the proof of Theorem \ref{discrete Isom}.
\end{proof}
Now we prove Theorem \ref{magic number}.
\begin{proof}[Proof of Theorem \ref{magic number}]
The proof is similar to that in \cite{FW2}. Let $X$ be the $\widetilde{M}$ with the lifted metric $h$ from $M$. By theorem \ref{discrete Isom}, the group $\Isom(X)$ is discrete. So $X/\Isom(X)$ is a compact orbifold covered by $M$. The degree of the cover is $d = [\Isom(\widetilde{M}): \pi_1(M)]$. 

Let $(\widehat{M}, h)$ be a finite sheeted cover of $(M, h)$ that is a nontrivial normal cover of $X/\Isom(X)$. Such cover of $M$ exists for group-theoretic reasons. Let $e$ be the degree of the cover $\widehat{M} \longrightarrow M$. %Let $W = \widehat{M}/\Isom(\widehat{M})$. 
Let $F = \Isom(X)/\pi_1(\widehat{M})$ be the group of deck transformations of $(\widehat{M}, h)$ as a covering space of $X/\Isom(X)$. 

Let $\mathcal{M}$ (respectively, $\widehat{\mathcal{M}}$) be the disjoint union of all the complete locally symmetric spaces corresponding to the pieces $Y_i$ in the cusp decomposition of $M$ (respectively, $\widehat{\mathcal{M}}$), i.e. they are the spaces before we delete the ends of  their cusps.  

By Lemma \ref{most symmetry}, the group $F$ is isomorphic to a subgroup $H$ of $\Out(\pi_1(\widehat{M}))$ since $F$ acts effectively on $(\widehat{M}, h)$. Then the action of each element of $H$ is realized by isometry of $\widehat{\mathcal{M}}$ with respect to the locally symmetric metric $g_{loc}$. Since $F$ is a finite group, $H$ cannot contain nontrivial \emph{twists} (see \cite{Tampwlocsym} for the definition of twists). Hence, the action of each element of $H$ on $\widehat{\mathcal{M}}$ is effective. 

We claim that if any element of $F$ preserves a component of $\mathcal{M}$ it acts effectively on that component. Suppose this is not true, i.e. there is $f\in F$ such that the restriction $f$ to a piece $M_i$ is the identity map. Since $f$ has finite order, the lift $\widetilde{f}$ of $f$ restricted to $\widetilde{M}_i$ has finite order. By Lemma \ref{finite order isom} below, this is not possible and the claim follows.

Let $\mathcal{W}$ be the quotient of $\widehat{\mathcal{M}}$ by $F$. %So $|F| \leq |\Out(\pi_1(\widehat{M}))|$. 
Now, $|H| = |F| = de$. So
\[ d = \dfrac{|H|}{e}  = \dfrac{\Vol(\widehat{\mathcal{M}}, g_{loc})}{e\Vol(\mathcal{W}, g_{loc})} = \dfrac{e\Vol(\mathcal{M}, g_{loc})}{e\Vol(\mathcal{W}, g_{loc})} = \dfrac{\Vol(\mathcal{M}, g_{loc})}{\Vol(\mathcal{W}, g_{loc})}. \]

Let $V_n > 0$ be a lower bound of the volume of $n$--dimensional orbifolds. Such a $V_n$ exists by a theorem of Kazhdan and Margulis (see \cite[Corollary XI.11.9]{Raghunathan}). Hence, 
\[d \leq \dfrac{\Vol(\mathcal{M},g_{loc})}{V_n} .\] 
So we can let the constant in the theorem to be $C =  \dfrac{\Vol(\mathcal{M},g_{loc})}{V_n}$.  
\end{proof}
\begin{figure}
\begin{center}
\includegraphics[height=80mm]{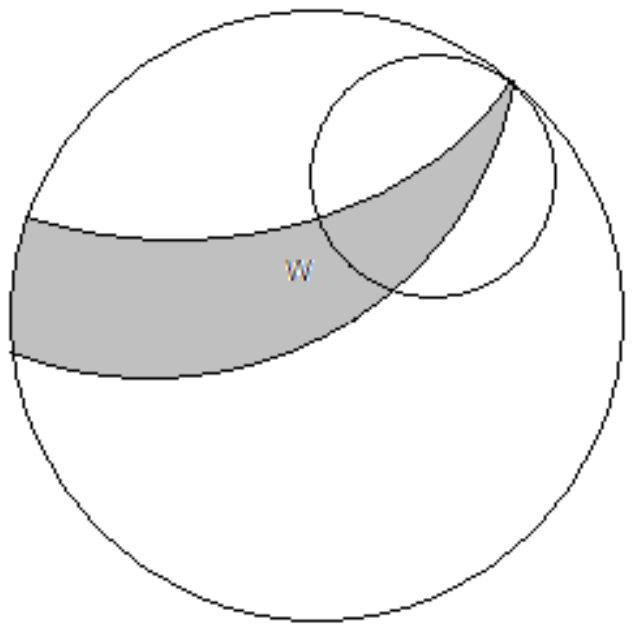}
\caption{The Fix set of $\widetilde{f}$.}
\end{center}
\end{figure}
\begin{lemma}\label{finite order isom}
Let $X$ be a noncompact locally symmetric space with finite volume. Let $f \colon X \longrightarrow X$ be a finite order isometry and $\widetilde{f} \colon \widetilde{X} \longrightarrow \widetilde{X}$ be a lift of $f$. Suppose that $f$ preserves a lowest-dimensional stratum $C$ of $\X$. If the restriction of $f$ to $C$ is homotopic to the identity map, then $\tilde{f}$ has infinite order.
\end{lemma}

\begin{proof}
Since $f$ is homotopic to the identity map on $C$, the lifted map $\widetilde{f}$ must be in the center of $\pi_1(C)$. %It follows that $\widetilde{f}$ belong to the unipotent radical of the parabolic corresponding to $\pi_1(C)$. 
Suppose that $\widetilde{f}$ has finite order. Since $\widetilde{M}$ is contractible, the set $\Fix(\widetilde{f})$ is nonempty. Since $\widetilde{f}$ preserves $\widetilde{C}$, a lift of $C$ that contains the basepoint, and $\widetilde{C}$ is contractible, it follows that the intersection $W = \Fix(\widetilde{f}) \cap \widetilde{C}$ is nonempty. 

We claim that $W$ is a contractible complete submanifold of $\widetilde{C}$. This is because $\Fix(\widetilde{f})$ is diffeomorphic to $\R^k\times W$, for $k$ the codimension of $C$ in $X$. Now, $\Fix(\widetilde{f})$ is contractible since it is convex (complete) submanifold of $\widetilde{M}$. Thus, $W$ must be contractible. Since $f$ preserves $C$ and is homotopic to the identity map on $C$, we have
\[ \widetilde{f}\circ\gamma = \gamma\circ\widetilde{f},\]
for all $\gamma \in \pi_1(C)$. This means that $\pi_1(C)$ preserves the Fix set of $\widetilde{f}$ in $\widetilde{C}$, which is $W$. So $\pi_1(C)$ acts freely, properly discontinuously on a contractible manifold of dimension at most $n-k-1$. This contradicts the cohomological dimension of $\pi_1(C)$, which is $n-k$.
\end{proof}

%\begin{proof}[Proof of Corollary \ref{weak rigidity}]
%Every isometry on $M$ lifts to an isometry on the universal cover $\widetilde{M}$ of $M$. Thus, if $\Out(\pi_1(M))$ is realized by a group of isometries of $M$, then $[\Isom(\widetilde{M}) : \pi_1(M)] > |\Out(\pi_1(M))|$. By Theorem \ref{magic number},  $[\Isom(\widetilde{M}) : \pi_1(M)] < \infty$. Hence $\Out(\pi_1(M))$ is finite. 
%\end{proof}

\bibliographystyle{amsplain}
\bibliography{bibliography}
%\begin{thebibliography}{9}

%\bibitem{FW1}
%B. Farb and S. Weinberger.
%Hidden symmetries and arithmetic manifolds
%\bibitem{FW2}
%B. Farb and S. Weinberger.
%Isometries, rigidity and universal covers.
%\bibitem{S}
%R. E. Schwartz.
%The quasi-isometry classification of rank on lattices, \textit{Publication mathematiques de l'I.H.E.S.}, tome 82 (1995), pp. 133-168.
%\bibitem{BGS}
%W. Ballmann, M. Gromov and V. Schroeder.
%Manifolds of Nonpositive Curvature.
%\bibitem{O}
%P. Ontaneda.
%The double of a hyperbolic manifold and non-positively curved exotic \textit{PL} struture. \textit{Transactions of the American Mathematical Society}, Vol 355, No.3, pp. 935-965.
%\bibitem{B}
%Bridson
%\bibitem{W-M}
%D. Witte Morris.
%\bibitem{E}
%D. Epstein
%\bibitem{GW}
%Gordon and Wilson.
%Isometry groups of Riemannian solvmanifolds.
%\bibitem{W}
%Wilson.
%Isometry groups on homogeneous nilmanifolds.
%\bibitem{A}
%Auslander.
%\bibitem{G}
%Gorbatsevich.
%On isometries of some Riemannian Lie groups.
%\bibitem{CR}
%Conner and Raymond.
%Manifolds with few periodic homeomorphisms.
%\bibitem{R}
%M. S. Raghunathan.
%Discrete Subgroups of Lie Groups. Ergebnisse de Mathematik, Springer-Verlag, New York 1972.
%\bibitem{M}
%G. Margulis.
%Discrete Subgroups of Semisimple Lie Groups. Springer-Verlag, New York 1991.
%%\bibitem{DS}
%Drutu and Sapir
%\bibitem{D}
%Dahnami
%\bibitem{N}
%Noskov.
%Word metrics on Lie groups and asymptotic cones of $2$-step nilpotent groups. 
%\bibitem{Tam}
%Nguyen Phan.
%Rigidity of piecewise locally symmetric manifolds.

%\end{thebibliography}
\end{document}